\documentclass[a4paper,12pt]{amsart}
\usepackage{amsmath,mathrsfs,amssymb,xcolor,MnSymbol}
\usepackage[top=3cm,bottom=3cm,outer=3cm,inner=3cm,marginpar=3cm]{geometry}
\usepackage[utf8]{inputenc}

\usepackage{hyperref}
\hypersetup{
	colorlinks = true,
	linkcolor = red,
	anchorcolor = black,
  	citecolor = blue,
	filecolor = red,
	urlcolor = magenta,
	pdfauthor=author
	menucolor = red
}

\newtheorem{theorem}{Theorem}[section]

\newtheorem{lemma}[theorem]{Lemma}
\newtheorem{corollary}[theorem]{Corollary}
\newtheorem*{corollary*}{Corollary}
\newtheorem{proposition}[theorem]{Proposition}

\theoremstyle{definition}

\newtheorem{remark}[theorem]{Remark}

\numberwithin{equation}{section}

\newcommand{\bB}{{\mathbb B} }

\newcommand{\paren}[1]{\left(#1\right)}

\newcommand{\pd}[2]{\frac{\partial#1}{\partial#2}}
\newcommand{\abs}[1]{\left\vert#1\right\vert}
\newcommand{\norm}[1]{\left\|#1\right\|}

\newcommand{\inner}[1]{\left\langle{#1}\right\rangle}

\newcommand{\set}[1]{\left\{#1\right\}}

\newcommand{\ip}[1]{\mathrm{Im}\;s}

\def\ol#1{{\overline{#1}}}

\def\ii{\mathbf{i}}

\def\sL{\mathsf{L}}

\def\KE{K\"ahler-Einstein }
\def\KM{\mathsf{M}}
\def\Hom{\mathrm{Hom}}

\def\RR{\mathbb{R}} 
\def\CC{\mathbb{C}} 
\def\DD{\mathbb{D}} 
\def\NN{\mathbb{N}} 

\def\iddbar{\ii\partial\bar\partial} %
\def\ddc{{dd^c}} %

\def\ric{\mathrm{Ric}}

\def\del{\partial}

\def\sL{\mathsf{L}}

\begin{document}
\title[K\"ahler Hyperbolicity Modulus]{K\"ahler Hyperbolicity Modulus for Simply-connected K\"ahler Hyperbolic manifolds}

\author{Young-Jun Choi}
\address{Department of Mathematics and Institute of Mathematical Science, Pusan National University, 2, Busandaehak-ro 63beon-gil, Geumjeong-gu, Busan, 46241, Republic of Korea}
\email{youngjun.choi@pusan.ac.kr}

\author{Kang-Hyurk Lee}
\address{Department of Mathematics and Research Institute of Natural Science, Gyeongsang National University, Jinju, Gyeongnam, 52828, Republic of Korea}
\email{nyawoo@gnu.ac.kr}


\thanks{The first named author was supported by the National Research Foundation of Korea (NRF) grant funded by the Korea government (No.~NRF-2023R1A2C1007227). 
The authors were supported by Samsung Science and Technology Foundation under Project Number SSTF-BA2201-01. }

\begin{abstract}
This paper investigates the K\"ahler hyperbolicity modulus on complete K\"ahler manifolds, with a particular focus on hyperconvex domains and bounded strongly pseudoconvex domains. 
Our main result establishes a lower bound for the K\"ahler hyperbolicity modulus in terms of the boundary behavior of the gradient length of a plurisubharmonic function.
As applications, we compute the K\"ahler hyperbolicity modulus for bounded symmetric domains.
Furthermore, we obtain lower bounds for the K\"ahler hyperbolicity modulus on bounded strongly pseudoconvex domains equipped with K\"ahler-Einstein metrics or Bergman metrics.
\end{abstract}

\maketitle

\section{Introduction}
A complete K\"ahler manifold $X$ is called to be \emph{K\"ahler hyperbolic} if its K\"ahler  form $\omega$ is \emph{$d$-bounded}, that is,  there exists a real $1$-form $\eta$ on $X$ such that 
\begin{equation*}
	d\eta=\omega
	\quad\text{and}\quad
	\norm{\eta}_{L^\infty}
	:=
	\sup_X \abs{\eta}_\omega<+\infty
\end{equation*}
where $\abs{\eta}_\omega$ denotes the pointwise length of $\eta$ with respect to $\omega$ (refer to \cite{Gromov1991} for the original definition). 

Due to \cite{Gromov1991}, every K\"ahler hyperbolic manifold satisfies the strong $L^2$-Lefschetz vanishing theorem:
\begin{equation*}
\dim\mathcal H^{p,q}_{(2)}(X) = 0 \quad \text{for } p+q \neq n,
\end{equation*}
where $\mathcal H^{p,q}_{(2)}(X)$ denotes the space  of square integrable harmonic $(p,q)$-forms with respect to $\omega$ and $n$ is the complex dimension of $X$ (see also \cite{Donnelly1994}). 

In \cite{Gromov1991}, this vanishing theorem was proved by deriving a lower bound for the spectrum of the Hodge Laplacian.
More precisely, this lower bound is expressed as a uniform constant, depending only on the degree and the dimension, multiplied by the reciprocal of the $L^\infty$-norm of the $1$-form $\eta$ (see \cite{Cho_Choi_Lee2026Ar} for explicit formulae). 
For instance, for any smooth function $\varphi$ on $X$ with compact support, we have
\begin{equation*}
    \frac{\int_X \abs{d\varphi}_\omega^2 dV}{\int_X \varphi^2dV}
    \ge
    \frac{n^2}{4\norm{\eta}^2_{L^\infty}}.
\end{equation*}
This inequality implies that $(X,\omega)$ has positive spectrum.
Specifically, the bottom of the $L^2$-spectrum of the Laplace–Beltrami operator, denoted by $\lambda_0$, satisfies
\begin{equation}\label{eqn:bottom_spectrum_estimate_eta}
	\lambda_0\ge \frac{n^2}{4\norm{\eta}^2_{L^\infty}}
\end{equation}
for any bounded $1$-form $\eta$ with $d\eta=\omega$. 
We now define the \emph{K\"ahler hyperbolicity modulus} of $(X,\omega)$ by
\begin{equation*}
	\KM:=\KM(X,\omega)
	:=\inf\set{\norm{\eta}^2_{L^\infty}:\eta\;\;\text{is a 1-form s.t.}\;\omega=d\eta}
\end{equation*}
Given $\KM>0$, we can sharpen \eqref{eqn:bottom_spectrum_estimate_eta} as follows:
\begin{equation}
\label{eqn:bottom_spectrum_estimate}
	\lambda_0\ge\frac{n^2}{4\KM}.
\end{equation}
The K\"ahler hyperbolicty modulus is a K\"ahler invariant which measures the degree of hyperbolicity of $(X,\omega)$: smaller values of $L$ correspond to stronger hyperbolic behavior.

It is remarkable that the estimate \eqref{eqn:bottom_spectrum_estimate} is sharp. 
More precisely, let the unit ball $\bB^n$ in $\CC^n$ equipped with the K\"ahler-Eistein metric $\omega$ with Ricci curvature $-(n+1)$. 
In this case, the K\"aher hyperbolicity modulus satisfies $\KM=1/2$ (see Section~\ref{sec:main_theorem}).
It then follows from~\eqref{eqn:bottom_spectrum_estimate} that
\begin{equation*}
	\lambda_0(\bB^n,\omega)\ge\frac{n^2}{4\KM}=\frac{n^2}{2}.
\end{equation*}
On the other hand, it is well that $\lambda_0(\bB^n,\omega)=n^2/2$ (see, e.g.,~\cite{Chavel-Book}).
Hence the estimate is sharp.

For a general $(X,\omega)$, it is not easy to get an explicit value of $\KM(X,\omega)$. If one has a concrete $1$-form $\eta$ and compute $\norm{\eta}_{L^\infty}$, then this immediately provides an upper bound of $\KM$: 
\begin{equation*}
	\KM\leq \norm{\eta}_{L^\infty}.
\end{equation*}
On the other hand, under a certain negativity condition of curvature, one can obtain an upper bound. 
More precisely, let $(X,\omega)$ be a simply connected K\"ahler manifold whose Riemannian sectional curvature is bounded from above by a negative constant $-K^2$.
Then there is a primitive $\eta$ of $\omega$ such that
\begin{equation*}
	\norm{\eta}_{L^\infty}^2\le\frac{n^2}{K}.
\end{equation*}
(See~\cite{Ballmann-Book}.)
In contrast, \eqref{eqn:bottom_spectrum_estimate} gives a lower bound of $\KM$ provided that the bottom of the $L^2$-spectrum is known. But the explicit computation of $\lambda_0$ is, in general, also a difficult problem.

%

\medskip

In this paper, we investigate the K\"ahler hyperbolic modulus on simply-connected complete K\"ahler manifolds, with particular emphasis on bounded pseudoconvex domains. 
More precisely, our aim is to find a lower bound of the K\"ahler hyperbolicity modulus by analyizing the boundary behavior of the plurisubharmonic exhaustion function.

\section{Lower bounds for K\"ahler hyperbolicity modulus}
\label{sec:main_theorem}
In this section, we prove the main theorem. 
To this end, we first introduce the following definition.

For a manifold $X$, a sequence $(x_j)\subset X$ is said to be \emph{compactly divergent}, written $x_j\to\infty$, if for any compact subset $K\subset X$, there is a positive integer $N$ such that $x_j\in X\setminus K$ for all $j\geq N$. 
For a function $f:X\to\RR$, put
\begin{equation*}
\liminf_{x\to\infty} f(x) = \inf \paren{\liminf_{j\to\infty} f(x_j)}
\end{equation*}
where the infimum is taken over all compactly divergent sequences of $X$.

\begin{theorem}
\label{thm:main_theorem}
Let $(X,\omega)$ be a simply-connected manifold equipped with a complete K\"ahler metric.
Suppose that $\omega$ admits a global potential $g:X\rightarrow\RR$ such that $\ddc g=\omega$, satisfying the following conditions:
\begin{enumerate}
\item $g$ is an exhaustion function on $X$;
\item There exists $C>0$ such that 
\begin{equation*}
\frac{1}{2}\liminf_{x\rightarrow\infty}\abs{\partial g}^2_\omega(x)
=
\liminf_{x\rightarrow\infty}\abs{d^cg}^2_\omega(x)\ge C.
\end{equation*}
\end{enumerate}
Then, for any  global $1$-form $\eta$ satisfying $d\eta=\omega$, we have
\begin{equation*}
	\norm{\eta}_{L^\infty}^2\ge C.
\end{equation*}
In particular, the K\"ahler hyperbolicity modulus satisfies $\KM(X,\omega)\ge C$.
\end{theorem}

\begin{proof}
Since $\eta$ and $d^cg$ satisfy 
\begin{equation*}
d\eta=\omega=d(d^cg),
\end{equation*} 
the difference $\eta-d^cg$ is $d$-closed.
Since $X$ is simply connected, $\pi_1(X)=0$.
By the de Rham theorem and the identification $H^1(X,\RR)\cong\Hom(\pi_1(X),\RR)$, it follows that $H^1_{dR}(X)=0$.
Hence there exists a smooth function $f:X\rightarrow\RR$ such that
\begin{equation*}
	\eta=df+d^cg.
\end{equation*}
Since the level set $L_j=\set{z\in X:g(z)=j}$ is compact, for each $j\in\NN$
there exists a point $p_j\in L_j$ such that
\begin{equation*}
	\sup_{L_j}f=f(p_j).
\end{equation*}
By the method of Lagrange multipliers, there exists $\lambda_j\in\RR$ such that 
$$df(p_j)=\lambda_j dg(p_j).
$$
We compute
\begin{align*}
	\abs{\eta}^2_\omega
	=
	\inner{df+d^cg,df+d^cg}
	=
	\abs{df}_\omega^2+2\inner{d^cg,df}+\abs{d^cg}_\omega^2.
\end{align*}
Evaluating this identity at $p_j$, we obtain 
\begin{equation*}
	\abs{\eta}^2_\omega
	=
	\abs{df}_\omega^2+2\lambda_j\inner{d^cg,dg}+\abs{d^cg}_\omega^2
\end{equation*}
Since $d^cg=J^*dg$ and $\omega$ is hermitian, the forms $dg$ and $d^cg$ are orthogonal.
Hence the mixed term vanishes, and
\begin{equation*}
	\abs{\eta}^2_\omega
	=
	\abs{df}_\omega^2+\abs{d^cg}_\omega^2
	\quad\text{at $p_j$}.
\end{equation*}
Consequently,
\begin{equation*}
\norm{\eta}_{L^\infty}^2\ge \liminf_{x\rightarrow\infty}\abs{\eta}^2_\omega(x)\ge C.
\end{equation*}
This completes the proof.
\end{proof}

We now illustrate Theorem~\ref{thm:main_theorem} by considering the unit ball equipped with its complete K\"ahler-Einstein metric.

Let $\bB^n=\set{z\in\CC^n: \abs{z}<1}$ be the unit ball, where $z=(z^1,\ldots,z^n)$ is the complex coordinates of $\CC^n$ and $\abs{z}^2 = \sum\abs{z^\alpha}^2$. 
The complete \KE metric of $\bB^n$ has the K\"ahler form $\omega=\sqrt{-1} h_{\alpha\bar\beta}dz^\alpha\wedge dz^{\bar\beta}$  given by
\begin{equation*}
h_{\alpha\bar\beta}
	=\frac{1}{\paren{1-\abs{z}^2}^2}
	\paren{
		\delta_{\alpha\bar\beta}(1-\abs{z}^2)
		+z^{\bar\alpha}z^\beta
	}\;.
\end{equation*}
It satisfies the Einstein condition $\mathrm{Ric}(\omega)=-(n+1)\omega$ and has the explicit value of $\lambda_0$:
\begin{equation*}
\lambda_0=\frac{n^2}{2}.
\end{equation*}
Let us consider the potential function 
\begin{equation}\label{eqn:model potential}
g=-\log(1-\abs{z}^2)
\end{equation}
of $\omega$, that is, $dd^cg=\omega$. One can easily see that
\begin{equation*}
\abs{d^c g}_\omega^2 = \frac{1}{2} \abs{z}^2 \quad\text{for $z\in\bB^n$}
\end{equation*}
so $\eta=d^c g$ satisfies $d\eta=\omega$ and $\norm{\eta}^2_{L^\infty}=1/2$. 
Thus the K\"ahler hyperbolicity modulus $\KM=\KM(\bB^n,\omega)$ is not greater than $1/2$:
\begin{equation*}
\KM\leq \frac{1}{2}.
\end{equation*}

On the other hand, the potential function $g$ satisfies the hypothesis in Theorem~\ref{thm:main_theorem}.
Hence for any $1$-form $\eta$ with $d\eta=\omega$ we have
\begin{equation*}
	\norm{\eta}^2_{L^\infty}\ge\frac{1}{2}.
\end{equation*}
by Theorem~\ref{thm:main_theorem}, which yields that $\KM\ge1/2$.
Therefore, $\KM=1/2$.
\medskip


An immediate corollary also follows.

\begin{corollary}
\label{cor:main_corollary}
Under the same hypotheses as in Theorem~\ref{thm:main_theorem}, assume further that
\begin{equation*}
\liminf_{x\rightarrow\infty}\abs{d^cg}^2_\omega(x)
=
\norm{d^cg}^2_{L^\infty}.
\end{equation*}
Then we have
\begin{equation*}
	\KM(X,\omega)=\norm{d^cg}^2_{L^\infty}.
\end{equation*}
\end{corollary}

\section{Application 1}

In this section, we compute the K\"ahler hyperbolicity modulus for the bounded symmetric domains and simply-connected bounded strongly pseudoconvex domains as an application of Theorem~\ref{thm:main_theorem}.

\subsection{Bounded symmetric domains}

Let $\Omega$ be an irreducible bounded symmetric domain and let $N_\Omega$ be its generic norm. 
Then the Bergman kernel $K_\Omega$ of $\Omega$ is of the form 
\begin{equation*}
K_\Omega(z,w) 
	= c N_\Omega(z,w)^{-c_\Omega}
\end{equation*}
for a normalizing constant $c$ by the Euclidean volume of $\Omega$ and the genus $c_\Omega$ of $\Omega$ which is a positive integer. 
The Bergman metric $\omega=dd^c\log K$ is the complete \KE metric with Ricci curvature $-1$, that is, $\ric(\omega)=-\omega$. 
In~\cite{Choi_Lee_Seo2025Ar}, the invariant $\sL_\Omega$ of $(\Omega,\omega)$ is defined by
\begin{equation*}
\sL_\Omega = \sqrt{rc_\Omega}
\end{equation*}
where $r$ is the rank of $\Omega$, the dimension of a maximal totally geodesic polydisc in $\Omega$. 
(In fact, this is called the \emph{K\"ahler-hyperbolicity length} in~\cite{Choi_Lee_Seo2025Ar}.) 
If a bounded symmetric domain $\Omega$ is biholomorphic to a product $\Omega_1\times\cdots\times\Omega_s$ of irreducible bounded symmetric domains, then we set
\begin{equation*}
\sL_\Omega 
	= \paren{\sum_{j=1}^s \sL_{\Omega_j}^2}^{1/2} 
	=\paren{ \sum_{j=1}^s r_j c_{\Omega_j} }^{1/2}
\end{equation*}
where each $r_j$ and $c_{\Omega_j}$ are the rank and the genus of $\Omega_j$.

The next theorem says that the K\"ahler hyperbolicity modulus in this paper essentially generalize the one in~\cite{Choi_Lee_Seo2025Ar}.
\begin{theorem}
\label{thm:main_application}
Let $\Omega$ be a bounded symmetric domain and $\omega$ be its Bergman metric.
For any global $1$-form $\eta$ with $\omega=d\eta$, we have
\begin{equation}
\label{eqn:lower_bound_eta}
	\norm{\eta}^2_{L^\infty}\ge\frac{\sL^2_\Omega}{2},
\end{equation}
In particular, the K\"ahler hyperbolic modulus of $(\Omega,\omega_B)$ is given by
\begin{equation*}
	\KM(\Omega,\omega_B)=\frac{\sL^2_\Omega}{2}.
\end{equation*}
\end{theorem}

The estimate \eqref{eqn:lower_bound_eta} was previously established in~\cite{Choi_Lee_Seo2025Ar} under the additional assumption that $\eta$ is $d^c$-closed.
The extra $d^c$-closedness condition was subsequently removed in~\cite{Cho_Choi_Lee2026Ar}.
In that work, the authors relied on McKean's lower bound estimate~\cite{McKean1970} for the bottom of the spectrum in Riemannian geometry and proved the result indirectly by a contrapositive argument.

In contrast, the current proof is elementary and direct, relying solely on Theorem~\ref{thm:main_theorem}.

\begin{proof}[Proof of Theorem~\ref{thm:main_application}]
	Let $\Omega$ be a bounded symmetric domain of rank $r$,
	and let $\omega$ be the Bergman metric, which is the complete \KE metric with Ricci curvature $-1$. 
	Let $\eta$ be a 1-form on $\Omega$ satisfying $d\eta=\omega$.
	
	By the argument of Section 4.1 in~\cite{Choi_Lee_Seo2025Ar}, there exists a totally geodesic holomorphic disc $\iota:\DD\hookrightarrow\Omega$ such that the pullback metric $\iota^*\omega$ has constant Gaussian curvature
	\begin{equation*}
		-\kappa=-\frac{2}{\sL_\Omega^2}<0.	
	\end{equation*}
	Hence $(\DD,\iota^*\omega)$ is the Poincar\'e metric of constant Gaussian curvature $-\kappa$, and therefore
	\begin{equation*}
		\iota^*\omega
		=
		\frac{2\ii dz\wedge d\ol{z}}{\kappa\paren{1-\abs{z}^2}^2}.
	\end{equation*}
	This metric admits a global potential $g:\DD\rightarrow\RR$ given by
	\begin{equation*}
		g=-\frac{1}{\kappa}\log\paren{1-\abs{z}^2}^2,
	\end{equation*}
	Then $g$ is a plurisubharmonic exhaustion function on $\DD$.
	Moreover, 
	\begin{align*}
		\abs{\partial g}^2_{\iota^*\omega}
		=
		\abs{\frac{2}{\kappa}\cdot\frac{\ol z}{1-\abs{z}^2}dz}^2_{\iota^*\omega}
		=
		\frac{2\abs{z}^2}{\kappa}
		\rightarrow\frac{2}{\kappa}
		\quad\text{as}\quad z\rightarrow\partial\DD.
	\end{align*}
	In particular,
	\begin{equation*}
		\liminf_{x\rightarrow\del\Omega}\abs{d^cg}^2_\omega(x)\ge\frac{1}{\kappa}=\frac{\sL_\Omega^2}{2}.
	\end{equation*}
	Since $\iota:\DD\rightarrow\Omega$ is an isometric embedding and $d\paren{\iota^*\eta}=\iota^*\omega$, Theorem~\ref{thm:main_theorem} yields
	\begin{equation*}
		\norm{\eta}^2_{L^\infty}
		=
		\sup_{\Omega}\abs{\eta}_\omega^2
		\ge
		\sup_{\iota(\DD)}\abs{\eta}_\omega^2
		=
		\sup_\DD\abs{\iota^*\eta}^2_{\iota^*\omega}\ge\frac{\sL_\Omega^2}{2}.
	\end{equation*}
	According to Theorem 1 in \cite{Kai_Ohsawa2007} and Theorem 1.1 in \cite{Choi_Lee_Seo2025Ar}, there exists a global potential $\varphi$ on $\Omega$ with $\omega=\ddc\varphi$ such that
	\begin{equation*}
		\abs{d^c\varphi}^2_\omega=\frac{\sL^2_\Omega}{2}.
	\end{equation*}
	This result yields the last assertion, concluding the proof.
\end{proof} 

\begin{remark}
The conclusion of Theorem~\ref{thm:main_application} does not follows from Corollary~\ref{cor:main_corollary}.
The proof of Theorem~\ref{thm:main_application} strongly relies on the geometry of bounded symmetric domains. In particular, it uses the following two geometric properties.
First, every bounded homogeneuous domain admits a global potential whose gradient length is constant~(\cite{Kai_Ohsawa2007}).
Second, any bounded symmetric domains contain a totally geodesic disc with maximal holomorphic sectional curvature~(\cite{Choi_Lee_Seo2025Ar}).
\end{remark}

\subsection{Strongly pseudoconvex domains with defining functions}
\begin{proposition}
\label{prop:strongly_pseudoconvex_domain}
Let $\Omega$ be a simply-connected bounded strongly pseudoconvex domain with a defining function $r$ satisfying $\iddbar r>0$ on $\ol\Omega$.
If $\omega$ be the complete K\"ahler metric defined by $\omega=-\iddbar\log(-r)$, then for any 1-form $\eta$ with $d\eta=\omega$ satisfies
\begin{equation*}
	\norm{\eta}^2_{L^\infty}\ge\frac{1}{2}.
\end{equation*}
In particular,
\begin{equation*}
	\KM(X,\omega)=\frac{1}{2}.
\end{equation*}
\end{proposition}

\begin{proof}
The global potential function $w:=-\log(-r)$ is clearly a strictly plurisubharmonic exhaustion function on $\Omega$.
A direct computation yields that 
\begin{equation*}
w_{\alpha\bar\beta}
=\frac{r_{\alpha\bar\beta}}{-r}+
\frac{r_\alpha r_{\bar\beta}}{r^2}\;,
\end{equation*}
and the inverse is 
\begin{equation*}
w^{\bar\beta\alpha}
	=(-r)\paren{r^{\bar\beta\alpha}
	+\frac{r^{\bar\beta}r^\alpha}{r-\abs{\partial r}^2}},
\end{equation*}
where
\begin{equation*}
(r^{\alpha\bar\beta})=
\paren{r_{\alpha\bar\beta}}^{-1},
\quad
r^\alpha
	=r^{\alpha\bar\beta}r_{\bar\beta},
\quad\text{and} \quad
\abs{\partial r}^2
	= r^{\alpha\bar\beta}r_\alpha r_{\bar\beta}.
\end{equation*}
Consequently, we have
\begin{equation}
\label{eqn:gradient_length_defining_function}
	\abs{d^cw}^2_\omega
	=
	\frac{1}{2}
	\abs{\partial w}^2_\omega
	=
	\frac{1}{2}
	w^{\bar\beta\alpha}w_\alpha w_{\bar\beta}
	=
	\frac{1}{2}\cdot\frac{\abs{\partial r}^2}{\abs{\partial r}^2-r}
	\nearrow
	\frac{1}{2}
	\quad\text{as}\quad x\rightarrow\partial\Omega,
\end{equation}
Therefore, Theorem~\ref{thm:main_theorem} implies the first assertion.
Since $\norm{d^cw}^2_{L^\infty}=1/2$, the second conclusion follows from Corollary~\ref{cor:main_corollary}.
\end{proof}

\begin{remark}
Li-Tran showed that the bottom of spectrum is exactly $\lambda_0(\Omega,\omega)=n^2/2$ under the assumption that $\Omega$ is a bounded strongly pseudoconvex domain equipped with the same metric defined as in Proposition~\ref{prop:strongly_pseudoconvex_domain} (\cite{Li_Tran2010}).
\end{remark}

\section{Application 2}

In this section, we discuss lower bounds for the K\"ahler hyperbolicity modulus of the K\"ahler-Einstein and Bergman metrics on simply connected bounded strongly pseudoconvex domains as an application of Theorem~\ref{thm:main_theorem}. 
We also treat the case of $2$-dimensional Thullen domains.

\subsection{K\"ahler-Einstein metrics}

\begin{theorem}
\label{thm:KE_case}
Let $\Omega$ be a simply-connected bounded strongly pseudoconvex domain in $\CC^n$ with smooth boundary equipped with K\"ahler-Einstein metric $\omega$ with Ricci curvature $-(n+1)$.
Then for any $1$-form $\eta$ satisfying $d\eta=\omega$, we have
\begin{equation*}
	\norm{\eta}^2_{L^\infty}\ge\frac{1}{2}.
\end{equation*}
In particular,
\begin{equation*}
	\KM(\Omega,\omega)\ge\frac{n^2}{2}.
\end{equation*}

\end{theorem}

\begin{proof}
By Theorem~\ref{thm:main_theorem}, it is enough to show that there exists a plurisubharmonic exhaustion function $g:\Omega\rightarrow\RR$ such that $\iddbar g=\omega$ and
\begin{equation*}
\liminf_{x\rightarrow\del\Omega}\abs{\del g}^2_\omega(x)\ge 1.
\end{equation*}
In fact, this function comes from the solution of complex Monge-Amp\`ere equation which yields the complete K\"ahler-Einstein metric on $\Omega$.
Then the boundary behavior of this solution, established by Cheng and Yau~\cite{Cheng_Yau1980}, gives the desired estimate.
Although this was already obtained in~\cite{Choi_Lee2023}, we give a sketch of the proof  for the readers' convenience.

Let $\Omega$ be a smooth bounded strongly pseudoconvex domain in $\CC^n$. 
It is well known that there exists a defining function $r\in C^\infty(\ol\Omega)$ such that the Levi form $\paren{r_{\alpha\bar\beta}}$ is strictly positive on $\overline\Omega$.
Then $\iddbar w$ where $w=-\log(-r)$ gives a complete K\"ahler metric on $\Omega$.
Moreover, the Ricci tensor of $w_{\alpha\bar\beta}$ is given by
\begin{equation}\label{E:compatible}
\begin{aligned}
R_{\alpha\bar\beta} 
	& =-\pd{^2}{z^\alpha\partial{z}^{\bar\beta}} \log\det(w_{\gamma\bar\delta}) \\
	& = -(n+1)w_{\alpha\bar\beta}
	-\frac{\partial^2}{\partial{z^\alpha}\partial{z^{\bar\beta}}}
	\log\det(r_{\gamma\bar\delta})(-r+\abs{\partial r}^2).
\end{aligned}
\end{equation}
If we denote by $F=\log\det(r_{\alpha\bar\beta})(-r+\abs{\partial r}^2)$, then $F$ is a positive smooth function in $\overline\Omega$ satisfying
\begin{equation*}
R_{\alpha\bar\beta}
+(n+1)w_{\alpha\bar\beta}
=
\frac{\partial^2F}{\partial{z^\alpha}\partial{z^{\bar\beta}}}.
\end{equation*}
According to the Cheng-Yau Theorem, there exists a unique solution of the following complex Monge-Ampere equation:
\begin{equation} \label{E:Monge-Ampere'}
\begin{aligned}
&\det(w_{\alpha\bar\beta}+u_{\alpha\bar\beta})
=e^{(n+1)u}e^F\det(w_{\alpha\bar\beta}) \\
&\quad\frac{1}{c}(w_{\alpha\bar\beta})\le(w_{\alpha\bar\beta}+u_{\alpha\bar\beta}) \le c(w_{\alpha\bar\beta}).
\end{aligned}
\end{equation}
Then it is easy to see that $g:=w+u$ is a potential of the unique complete \KE metric on $\Omega$.
In particular, $g$ is a plurisubharmonic exhaustion function.
In \cite{Fefferman1974}, Fefferman constructed a defining function $r$ satisfying 
\begin{equation*}
F
=
\log\det(r_{\alpha\bar\beta})(-r+\abs{\partial r}^2)
=
O(\abs{r}^{n+1})
\end{equation*}
Using this defining function, Cheng and Yau computed the boundary behavior of $u$.

\begin{theorem}[Simple Version \cite{Cheng_Yau1980}] \label{T:CY2}
Let $\Omega$ be a smooth strongly pseudoconvex domain in $\CC^n$ and let $r$ be a smooth defining function of $\Omega$. Suppose that $F=O(\abs r^{n+1})$ and $u$ is a solution of \eqref{E:Monge-Ampere'}. Then
\begin{equation*}
\abs{D^pu}(x)=O(\abs r^{n+1/2-p-\varepsilon})
\quad\text{for}\quad \varepsilon>0
\end{equation*}
where $\abs{D^pu}(x)$ is the Euclidean length of the $p$-th derivative of $u$.
\end{theorem}
In particular Theorem \ref{T:CY2} says that
\begin{equation} \label{E:boundary}
\abs{u_\alpha}= O(\abs r^{n-1/2-\varepsilon})
\quad\text{and}\quad
\abs{u_{\bar\beta}}= O(\abs r^{n-1/2-\varepsilon})
\end{equation}
for $\varepsilon>0$ and $1\le\alpha,\beta\le{n}$.

Before computing the boundary behavior of $\abs{\partial g}^2_\omega$, we introduce the following lemma.
\begin{lemma}[Lemma 5.3 in \cite{Choi2015A}] \label{L:identity}
There exists a hermitian $n\times{n}$ matrix 
\begin{equation*}
N=(N_{\alpha\bar\beta})
	\in\mathrm{Mat}_{n\times{n}}\paren{C^\infty(\Omega)\cap{C}^{n-3/2-\varepsilon}
	(\overline \Omega)}
\end{equation*}
with $\norm{N}=O(\abs r^{n-3/2-\varepsilon})$ for $\varepsilon>0$, which satisfies that
\begin{equation*}
g^{\bar\beta\alpha}-w^{\bar\beta\alpha}
	=w^{\bar\beta\gamma}N_{\gamma\bar\delta}w^{\bar\delta\alpha}.
\end{equation*}
In particular, $g^{\bar\beta\alpha}\in{C}^\infty(\Omega)\cap{C}^{n-3/2-\varepsilon}(\overline\Omega)$ and $g^{\bar\beta\alpha}=O(\abs r)$ for $\varepsilon>0$.
\end{lemma}

Now we consider the boundary behavior of $\abs{\partial\varphi}_\omega^2$ near the boundary.
By Lemma \ref{L:identity}, we have
\begin{align*}
\abs{\partial g}^2_\omega
&=
g_\alpha g_{\bar\beta}g^{\alpha\bar\beta}
=
(w+u)_\alpha(w+u)_{\bar\beta}g^{\alpha\bar\beta}
\\
&=
(w+u)_\alpha(w+u)_{\bar\beta}
\paren{w^{\bar\beta\alpha}+w^{\bar\beta\gamma}N_{\gamma\bar\delta}w^{\bar\delta\alpha}}
\\
&=
w_\alpha w_{\bar\beta}w^{\alpha\bar\beta}
	+
	w_\alpha u_{\bar\beta}w^{\bar\beta\alpha}
	+
	u_\alpha w_{\bar\beta}w^{\bar\beta\alpha}
	+
	(w+u)_\alpha(w+u)_{\bar\beta}
	w^{\bar\beta\gamma}N_{\gamma\bar\delta}w^{\bar\delta\alpha}.
\end{align*}
From \eqref{eqn:gradient_length_defining_function}, \eqref{E:boundary}, and Lemma~ \ref{L:identity}, we have the following boundary estimates:
\begin{equation*}
w_\alpha=O(\abs r^{-1}),\;\;
w^{\bar\beta\alpha}=O(\abs r),\;\;
u_\alpha=O(\abs r^{n-1/2-\varepsilon}),
\quad\text{and}\quad
N^{\gamma\bar\delta}
=O(\abs r^{n-3/2-\varepsilon}).
\end{equation*}
Consequently,
\begin{equation*}
\abs{\partial g}^2_\omega
=
\frac{\abs{\partial r}^2}{\abs{\partial r}^2-r}
+O(\abs r).
\end{equation*}
In particular,
\begin{equation*}
\abs{\partial g}^2_\omega\rightarrow1
\quad
\text{as}
\quad 
z\rightarrow\partial\Omega.
\end{equation*}
This completes the proof.
\end{proof}

\subsection{Bergman metrics}

\begin{theorem}
Let $\Omega\subset\CC^n$ be a simply-connected bounded strongly pseudoconvex domain with smooth boundary, equipped with the Bergman metric $\omega$.
For any $1$-form $\eta$ satisfying $d\eta=\omega$, the following estimate holds:
\begin{equation*}
	\norm{\eta}^2_{L^\infty}\ge\frac{n+1}{2}.
\end{equation*}
In particular,
\begin{equation*}
	\KM(\Omega,\omega)\ge\frac{n^2}{2(n+1)}.
\end{equation*}
\end{theorem}

\begin{proof}
By Theorem~\ref{thm:main_theorem}, it suffices to prove the existence of a plurisubharmonic exhaustion function $g:\Omega\rightarrow\RR$ satisfying $\iddbar g=\omega$ and
\begin{equation*}
\liminf_{x\rightarrow\del\Omega}\abs{\del g}^2_\omega(x)\ge 1.
\end{equation*}
Fefferman's asymptotic expansion of the Bergman kernel near strongly pseudoconvex boundary points \cite{Fefferman1974} implies that there exist $\varphi,\psi\in C^\infty(\ol\Omega)$ such that the Bergman kernel $K$ can be expressed as
\begin{equation*}
	K=\frac{\varphi}{(-r)^{n+1}}+\psi\log(-r),
\end{equation*}
where $r$ is a defining function for $\Omega$.
Since we are looking near the boundary, we may assume that $\varphi>0$.
In terms of a new defining function $\tilde r:=r\varphi^{\frac{1}{n+1}}$, we can rewrite $K$ as
\begin{equation*}
	K=
	\frac{1}{\paren{-\tilde r}^{n+1}}
	+
	\psi\log\paren{-\tilde r}\cdot\paren{\varphi^{\frac{1}{n+1}}}^{-1}
\end{equation*}
Consequently, the potential function $g:=\log K$ of the Bergman metric satisfies
\begin{equation*}
	g
	=
	-(n+1)\log(-r)
	+
	\log\paren{1+(-r)^{n+1}\paren{\psi\log(-r)-\frac{1}{n+1}\log\varphi}}
	=:
	(n+1)w+u.
\end{equation*}
It is straightforward to verify that $u$ satisfies the boundary estimate in Theorem~\ref{T:CY2}.
Therefore, applying the same argument as in the proof of theorem~\ref{thm:KE_case}, we have 
\begin{equation*}
\abs{\partial g}^2_\omega(x)\rightarrow(n+1)
\quad
\text{as}\quad x\rightarrow\partial\Omega,
\end{equation*}
which completes the proof.
\end{proof}

\subsection{Thullen domains in $\CC^2$}

Let $E_m$ be the Thullen domain in $\CC^2$ defined by
\begin{equation*}
	E_m:=\set{(z,w):\abs{z}^2+\abs{w}^{2m}<1},
\end{equation*}
which is a weakly pseudoconvex domain in $\CC^2$.
Then the function $r(z)=\abs{w}^2-\paren{1-\abs{z}^2}^{1/m}$ is a strictly plurisubharmonic defining function such that 
$$
r \in C^\infty(E_m)\bigcap C^{1/m}(\ol{E_m}).
$$
If we denote by $g:=-\log(-r)$, then the computation in~\cite{Li_Tran2010} says that 
\begin{equation*}
	\abs{\partial g}^2_\omega
	=
	1+r\cdot\frac{m-\abs{z}^2}{m\paren{1-\abs{z}^2}^{1/m}}
\end{equation*}
Then
\begin{eqnarray*}
	\abs{\partial g}^2_\omega
	&=&
	1+
	\paren{\abs{w}^2-\paren{1-\abs{z}^2}^{1/m}}
	\cdot\frac{m-\abs{z}^2}{m\paren{1-\abs{z}^2}^{1/m}}
	\\
	&\ge&
	1-\paren{1-\abs{z}^2}^{1/m}
	\cdot\frac{m-\abs{z}^2}{m\paren{1-\abs{z}^2}^{1/m}}
	\\
	&=&
	1-\frac{m-\abs{z}^2}{m},
\end{eqnarray*}
which implies that
\begin{equation*}
\liminf_{x\rightarrow\del\Omega}\frac{1}{2}\abs{dg}^2_\omega(x)
=
\liminf_{x\rightarrow\del\Omega}\abs{\del g}^2_\omega(x)
\ge\frac{1}{m}.
\end{equation*}
Therefore, Theorem~\ref{thm:main_theorem} implies that the K\"ahler hyperbolicity modulus satisfies that
\begin{equation*}
\KM(E_m,\iddbar g)\ge1/m.
\end{equation*}


\end{document}